\documentclass[12pt,reqno]{amsart}

\usepackage{amssymb,latexsym}

\usepackage{enumerate}

\usepackage[french,english]{babel}
\usepackage{amsmath}
\usepackage{graphicx}
\usepackage{amssymb}
\usepackage{bbm}
\usepackage{amsthm,mathtools}
\usepackage{ulem}
\usepackage{geometry}
\usepackage{tikz-cd}
\usepackage{mathrsfs}
\usepackage[colorinlistoftodos]{todonotes}
\usepackage{enumitem}
\usepackage{verbatim}
\usepackage[foot]{amsaddr}
\usepackage{dsfont}

\makeatletter

\@namedef{subjclassname@2010}{
	
	\textup{2020} Mathematics Subject Classification}

\makeatother
\newtheorem{thm}{Theorem}[section]
\newtheorem*{thm*}{Theorem}

\newtheorem{lem}[thm]{Lemma}

\theoremstyle{definition}

\numberwithin{equation}{section}

\newcommand{\M}{{\mathcal M}}

\newcommand{\F}{\mathcal F}

\usepackage{hyperref}
\hypersetup{hypertex=true,colorlinks=true,linkcolor=blue,anchorcolor=blue,citecolor=blue}
\newcommand{\newabstract}[1]{%
	\par\bigskip
	\csname otherlanguage*\endcsname{#1}%
	\csname captions#1\endcsname
	\item[\hskip\labelsep\scshape\abstractname.]
}

\begin{document}

	\baselineskip=17pt

	\title[Large Values of  quadratic character sums]{Large values of quadratic character sums}
	
	\author{Zikang Dong}
\author{Ruihua Wang}
\author{Weijia Wang}
\author{Hao Zhang}
\author{Shengbo Zhao}
\address[Zikang Dong]{School of Mathematical Sciences, Soochow University, Suzhou 215006, P. R. China}
\address[Ruihua Wang]{School of Fundamental Sciences, Hainan Bielefeld University of Applied Sciences, Danzhou 578101, P. R. China}
\address[Weijia Wang]{Morningside Center of Mathematics, Academy of Mathematics and Systems Science, Chinese Academy of Sciences, Beijing 100190, P. R. China}
\address[Hao Zhang]{School of Mathematics, Hunan University, Changsha 410082, P. R. China}
	\address[Shengbo Zhao]{4. School of Mathematical Sciences, Key Laboratory of Intelligent Computing and Applications (Tongji University), Ministry of Education, Tongji University, Shanghai 200092, China}
	
   \email{zikangdong@gmail.com}
\email{ruih.wan9@gmail.com}
\email{weijiawang@amss.ac.cn}
\email{zhanghaomath@hnu.edu.cn}
	\email{shengbozhao@hotmail.com}

	\date{\today}
	
	\begin{abstract} 
		In this paper, we investigate  large values of quadratic Dirichlet character sums.  We prove new Omega results for both short and long quadratic character sums under the assumption of the Generalized Riemann Hypothesis (GRH), which improve the previous results.
	\end{abstract}

	
	\maketitle
	\tableofcontents
    \section{Introduction}
The study of analytic properties of character sums is a classical topic in
analytic number theory. Let $q$ be a prime, and $\chi({\rm mod}\;q)$ be any primitive character. In 1918, P\'olya and Vinogradov showed the famous upper bound
$$\sum_{n\le x}\chi(n)\ll \sqrt q\log q.$$
This was improved to $\sqrt q\log_2q$ under GRH by Montgomery and Vaughan in 1977. Here and throughout, we use $\log_j$ to denote the $j$-th iterated logarithm. Before that, 
 Paley \cite{Paley} showed that there exists quadratic characters $\chi({\rm mod}\;q)$ and $x\le q$ such that $$|\sum_{n\le x}\chi(n)|\gg\sqrt q\log_2q.$$
 Thus the large values of character sums are also very important, especially of quadratic characters.
 
 Granville and Soundararajan \cite{GS01} made very important works on this subject, and some of their results were subsequently improved by \cite{BT,Hou,Mun}. These improvements are based on the development of the resonance method, which was firstly used to produce large values of zeta and $L$-functions by Hilberdink \cite{Hil} and Soundararajan \cite{Sound}.

Now we focus on the quadratic character sums. Let $\mathcal F$ denote the set of fundamental discriminants, and for
$d\in\mathcal F$ let $\chi_d$ be the associated primitive quadratic
Dirichlet character. For $y\ge1$, write
\[
S_d(y):=\sum_{n\le y}\chi_d(n).
\]

We refer to \cite{DWZ,Lam} for the distribution of $S_d(y)$.
In the present paper we are concerned with the quadratic family
\[
\{\chi_d:X<|d|\le2X,\ d\in\mathcal F\},
\]
and with character sums whose length is large compared with the ranges
usually accessible by direct orthogonality arguments. The quadratic
family is especially interesting in this respect. Unlike the full
family of Dirichlet characters modulo a fixed modulus, quadratic
characters do not enjoy an exact orthogonality relation. Consequently,
the analysis of a resonating moment requires a sufficiently precise
average of
\[
\chi_d(n)
\]
over fundamental discriminants. Under the Generalized Riemann
Hypothesis, Darbar and Maiti~\cite{DM} obtained such an
estimate with an error term depending mildly on the square-free part
of $n$. This makes it possible to use considerably longer resonators
in the quadratic family.

Recent work has produced a number of lower bounds for long quadratic
character sums. In the intermediate range
\[
\exp\left\{
4\sqrt{\log X\log_2X}\,\log_3X
\right\}
<
x
<
\exp\left\{(\log X)^{1/2+\varepsilon}\right\},
\]
Dong and Zhang~\cite{DZ} proved, under GRH, that
\[
\max_{\substack{X<|d|\le2X\\d\in\mathcal F}}
\frac{1}{\sqrt{x}}
S_d(x)
\ge
\exp\left\{
(1+o(1))
\sqrt{\frac{\log X}{\log_2X}}
\right\}.
\]
For the corresponding long dual sums, related lower bounds were
obtained in~\cite{DSWZ}. In a longer range,
\[
\exp\left\{(\log X)^{1/2+\varepsilon}\right\}
<
x
<
X^{1/2},
\]
Dong and Zhang~\cite{DZ} obtained
\[
\max_{\substack{X<|d|\le2X\\d\in\mathcal F}}
\frac{1}{\sqrt{x}}
S_d(x)
\ge
\exp\left\{
(1+o(1))
\sqrt{
\frac{
\log(\sqrt X/x)\log_3(\sqrt X/x)
}{
\log_2(\sqrt X/x)
}
}
\right\},
\]
while lower bounds for the corresponding dual character sums were
established in~\cite{DWWZ}.

The purpose of this paper is to improve these estimates by modifying
the resonating moment. A key feature of our argument is that, instead
of considering a moment in which the character sum is squared, we use
the first weighted moment
\[
\sum_{\substack{X<|d|\le2X\\d\in\mathcal F}}
S_d(x)R(d)^2,
\]
where $R(d)$ is a suitable non-negative resonator. After averaging over
$d$, the square condition in the main term is
\[
kmn=\square.
\]
When the resonator is supported on square-free integers, this condition
has the particularly simple form
\[
k=\frac{[m,n]}{(m,n)}\ell^2.
\]
Consequently, the main term contains the kernel
\[
\sqrt{x}\sqrt{\frac{(m,n)}{[m,n]}}.
\]
Thus the critical G\'al sum appears directly in the first moment, and
no square root has to be taken at the end of the resonance argument.
This observation is responsible for the improvement in the constants
obtained below.

Our first result concerns the intermediate range.

\begin{thm}\label{thm1}
Assume GRH. Let $X$ be sufficiently large and let
\[
\exp\left\{
4\sqrt{\log X\log_2X}\,\log_3X
\right\}
<
x
<
\exp\left\{(\log X)^{1/2+\varepsilon}\right\}.
\]
Then
\[
\max_{\substack{X<|d|\le2X\\d\in\mathcal F}}
\frac{1}{\sqrt{x}}
\sum_{n\le x}\chi_d(n)
\ge
\exp\left\{
(\sqrt2+o(1))
\sqrt{\frac{\log X}{\log_2X}}
\right\}.
\]
Moreover,
\[
\max_{\substack{X<|d|\le2X\\d\in\mathcal F}}
\frac{1}{\sqrt{|d|/x}}
\sum_{n\le |d|/x}\chi_d(n)
\ge
\exp\left\{
(\sqrt2+o(1))
\sqrt{\frac{\log X}{\log_2X}}
\right\}.
\]
\end{thm}

The proof of Theorem~\ref{thm1} uses a Hough-type multiplicative
resonator. The relevant Euler product produces the gain
\[
\exp\left\{
(2+o(1))
\sqrt{\frac{\log Y}{\log_2Y}}
\right\}
\]
for a resonator of effective length $Y=X^{1/2-o(1)}$. Since
\[
\log Y\sim\frac12\log X,
\]
the constant $2$ at the resonator level becomes $\sqrt2$ in the final
estimate. The lower restriction on $x$ allows us to localize the
G\'al mass to pairs satisfying
\[
\frac{[m,n]}{(m,n)}\ll x
\]
by Rankin's trick.

The second assertion of Theorem~\ref{thm1} is obtained from the same
resonator after applying P\'olya's Fourier expansion. Restricting to
negative fundamental discriminants, for which $\chi_d$ is odd, the
Fourier expansion naturally introduces the non-negative weight
\[
1-\cos\left(\frac{2\pi k}{x}\right).
\]
The additional factor $1/k$ in the Fourier sum permits a resonator of
length $X^{1/2-o(1)}$, and hence the same constant $\sqrt2$ is retained
in the dual range.

For still longer sums, the appropriate extremal object is the
square-free G\'al sum studied by de la Bret\`eche and
Tenenbaum~\cite{BT}. Their result gives
\[
\max_{\substack{|\mathcal M|=N\\
m\ {\rm squarefree}\ (m\in\mathcal M)}}
\frac{1}{N}
\sum_{m,n\in\mathcal M}
\sqrt{\frac{(m,n)}{[m,n]}}
=
\exp\left\{
(2+o(1))
\sqrt{
\frac{\log N\log_3N}{\log_2N}
}
\right\}.
\]
Combining this construction with the same first-moment idea gives our
second result.

\begin{thm}\label{thm2}
Assume GRH. Let $X$ be sufficiently large and suppose that
\[
\exp\left\{(\log X)^{1/2+\varepsilon}\right\}
<
x
<
X^{1/2},
\]
with
\[
\frac{\sqrt X}{x}\longrightarrow\infty.
\]
Then
\[
\max_{\substack{X<|d|\le2X\\d\in\mathcal F}}
\frac{1}{\sqrt{x}}
\sum_{n\le x}\chi_d(n)
\ge
\exp\left\{
(2+o(1))
\sqrt{
\frac{
\log(\sqrt X/x)\log_3(\sqrt X/x)
}{
\log_2(\sqrt X/x)
}
}
\right\}.
\]
Moreover,
\[
\max_{\substack{X<|d|\le2X\\d\in\mathcal F}}
\frac{1}{\sqrt{|d|/x}}
\sum_{n\le |d|/x}\chi_d(n)
\ge
\exp\left\{
(2+o(1))
\sqrt{
\frac{
\log(\sqrt X/x)\log_3(\sqrt X/x)
}{
\log_2(\sqrt X/x)
}
}
\right\}.
\]
\end{thm}

The appearance of
\[
\frac{\sqrt X}{x}
\]
in Theorem~\ref{thm2} reflects the maximal effective size of the
square-free resonating set allowed by the quadratic character
mean-value estimate. The constant $2$ is precisely the constant in the
square-free G\'al-sum extremal problem. Since our resonating moment is
linear in the character sum, this constant survives unchanged in the
final lower bound.

We briefly describe the main ingredients of the proofs. First, we use
the GRH conditional average over fundamental discriminants from
Darbar--Maiti~\cite{DM}. If
\[
n=n_0n_1^2
\]
with $n_0$ square-free, their estimate separates the square main term
from a sufficiently small error depending on $n_0$ and $n_1$. This is
essential for controlling the off-diagonal terms produced by the long
resonators.

Second, in the range of Theorem~\ref{thm1} we use a multiplicative
resonator of Hough type~\cite{Hou}. The corresponding G\'al quotient
can be evaluated by an Euler product, while Rankin's trick shows that
the contribution of pairs with excessively large
\[
[m,n]/(m,n)
\]
is negligible.

Third, in the range of Theorem~\ref{thm2} we replace the Hough
resonator by the square-free extremal sets of de la Bret\`eche and
Tenenbaum~\cite{BT}. The identity
\[
kmn=\square
\quad\Longleftrightarrow\quad
k=\frac{[m,n]}{(m,n)}\ell^2
\]
then converts the square contribution of the quadratic character
average directly into the critical G\'al sum.

Finally, the dual estimates are obtained by combining the resonance
method with P\'olya's Fourier expansion. In particular, after
restricting to odd quadratic characters, the Fourier kernel becomes
non-negative on a suitable range, allowing us to retain the same G\'al
structure as in the direct sums.

The remainder of the paper is organized as follows. In
Section~2 we collect the auxiliary results needed in the sequel,
including P\'olya's Fourier expansion, the conditional quadratic
character mean-value estimate, and the relevant G\'al-sum estimates.
Sections~3 and~4 are devoted to the two assertions of
Theorem~\ref{thm1}. In Sections~5 and~6 we prove the corresponding
statements of Theorem~\ref{thm2}.

\section{Preliminary Lemmas}
\begin{lem}\label{lem2.1}
 Let $\chi({\rm mod}\;q)$ be any primitive character and $0<\alpha<1$. Then we have
$$\sum_{n\le \alpha q}\chi(n)=\frac{\tau(\chi)}{2\pi i}\sum_{1\le|m|\le z}\frac{\overline\chi(m)}{m}(1-e(-\alpha m))+O(1+q\log q/z),$$
where $\tau(\chi):=\sum_{n\le q}\chi(n)e(n/q)$ is the Gauss sum and $e(a):=e^{2\pi ia}$.
\end{lem}
\begin{proof}
This is \cite[p.311, Eq. (9.19)]{MVbook}.
\end{proof}
The following conditional estimate for characters has a good error term in use.
    \begin{lem}\label{lem2.2}
	Assuming GRH. Let $n=n_0n_1^2$ be a positive integer with $n_0$ the square-free part of $n$.
	 Then for any $\varepsilon>0$, we obtain
	\begin{align*}
	\sum_{|d|\le X\atop d\in\F} \chi_{d}(n)=\frac{X}{\zeta(2)}\prod_{p|n}\frac{p}{p+1}{1}_{n=\square}+ O\left(X^{\frac12+\varepsilon}f(n_0)g(n_1)\right),
	\end{align*}
	where   ${{1}}_{n=\square}$ indicates the indicator function of the square numbers, and
    $$f(n_0)=\exp((\log n_0)^{1-\varepsilon}),\;\;\;\;g(n_1)=\sum_{d|n_1}\frac{\mu(d)^2}{d^{\frac12+\varepsilon}}.$$
\end{lem}
\begin{proof}
    This follows directly from Lemma 1 of \cite{DM}.
\end{proof}
On the one hand, it is clear that 
$$f(n_0)\le n_0^\varepsilon\le n^\varepsilon,\;\;\;\;g(n_1)\le n_1^\varepsilon\le n^\varepsilon.$$
On the other hand, if we denote the largest prime factor of $n$ by $P_+(n)$, then $n_0,n_1\le \prod_{p\le P_+(n)}p$.
So easily we have
$$f(n_0)\le\exp\big(P_+(n)^{1-\varepsilon}\big),\;\;\;\;\;g(n_1)\le\exp\big(P_+(n)^{\frac12-\varepsilon}\big).$$

\begin{lem}\label{GCD}
    Let $\M$ be any set of positive squarefree integers with $|\M|=N$. Then as $N\to\infty$, we have
    $$\max_{|\M|=N}\sum_{m,n\in\M}\sqrt{\frac{(m,n)}{[m,n]}}=N\exp\bigg((2+o(1))\sqrt{\frac{\log N\log_3N}{\log_2N}}\bigg).$$
\end{lem}
\begin{proof}
    This is Eq. (1.5) of \cite{BT}.
\end{proof}
Note that in the proof of the above lemma, the choice for the set $\M$ satisfies $y_\M:=\max_{m\in\M}P_+(m)\le (\log N)^{1+o(1)}$.

\section{Proof of the first part of Theorem \ref{thm1}}

We prove that, under GRH, uniformly in the range
\[
4\sqrt{\log X\log_2X}\,\log_3X
<
\log x
<
(\log X)^{1/2+\varepsilon},
\]
we have
\begin{equation*}
\max_{\substack{X<|d|\le2X\\ d\in\mathcal F}}
\frac{1}{\sqrt{x}}
\sum_{n\le x}\chi_d(n)
\ge
\exp\left\{
(\sqrt{2}+o(1))
\sqrt{\frac{\log X}{\log_2X}}
\right\}.
\end{equation*}

Fix sufficiently small constants
\[
0<\delta<\frac{1}{100},
\qquad
0<\rho<\frac{1}{100},
\]
and put
\begin{equation*}
Y:=\frac{X^{1/2-\delta}}{x}.
\end{equation*}
Since
\[
\log x=o(\log X),
\]
we have
\begin{equation}
\log Y
=
\left(\frac12-\delta+o(1)\right)\log X
\label{eq:logY}
\end{equation}
and
\begin{equation}
\log_2Y=\log_2X+O(1).
\label{eq:log2Y}
\end{equation}

Put
\begin{equation}
\lambda
:=
\sqrt{(1-\rho)\log Y\log_2Y}.
\label{eq:lambda-def}
\end{equation}
We define a non-negative multiplicative function $r$ supported on
square-free integers by setting
\begin{equation}
r(p)
=
\frac{\lambda}{\sqrt p\log p}
\label{eq:r-p}
\end{equation}
for
\[
\lambda^2\le p\le \exp\{(\log\lambda)^2\},
\]
and
\[
r(p)=0
\]
for all other primes. Moreover,
\[
r(p^\nu)=0
\qquad(\nu\ge2).
\]
Denote
\[
\mathcal P
:=
\left\{
p:
\lambda^2\le p\le
\exp\{(\log\lambda)^2\}
\right\}.
\]

We shall first establish the G\'al-type estimate associated with this
resonator.

\subsection{The G\'al quotient}

Consider
\[
\mathcal G(r)
:=
\frac{
\displaystyle
\sum_{m,n\ge1}
r(m)r(n)
\sqrt{\frac{(m,n)}{[m,n]}}
}{
\displaystyle
\sum_{n\ge1}r(n)^2
}.
\]
Since $r$ is multiplicative and supported on square-free integers,
Euler factorization gives
\begin{align*}
\sum_{m,n\ge1}
r(m)r(n)
\sqrt{\frac{(m,n)}{[m,n]}}
&=
\prod_{p\in\mathcal P}
\left(
1+r(p)^2+\frac{2r(p)}{\sqrt p}
\right),
\end{align*}
whereas
\begin{equation*}
\sum_{n\ge1}r(n)^2
=
\prod_{p\in\mathcal P}
(1+r(p)^2).
\end{equation*}
Therefore
\begin{equation*}
\mathcal G(r)
=
\prod_{p\in\mathcal P}
\left(
1+
\frac{2r(p)}
{\sqrt p(1+r(p)^2)}
\right).
\end{equation*}

Since
\[
r(p)\ll\frac1{\log\lambda}=o(1),
\]
we have
\begin{align*}
\log\mathcal G(r)
&=
2\sum_{p\in\mathcal P}
\frac{r(p)}{\sqrt p}
+
o\left(
\sqrt{\frac{\log Y}{\log_2Y}}
\right)
\nonumber\\
&=
2\lambda
\sum_{p\in\mathcal P}
\frac1{p\log p}
+
o\left(
\sqrt{\frac{\log Y}{\log_2Y}}
\right).
\end{align*}
By the prime number theorem and partial summation,
\begin{align}
\sum_{\lambda^2\le p\le
\exp\{(\log\lambda)^2\}}
\frac1{p\log p}
&=
\frac{1+o(1)}{2\log\lambda}.
\label{eq:prime-sum}
\end{align}
Consequently,
\[
\log\mathcal G(r)
=
(1+o(1))
\frac{\lambda}{\log\lambda}.
\]
Since
\[
\log\lambda
=
\left(\frac12+o(1)\right)\log_2Y,
\]
it follows from \eqref{eq:lambda-def} that
\begin{equation*}
\log\mathcal G(r)
=
\left(
2\sqrt{1-\rho}+o(1)
\right)
\sqrt{\frac{\log Y}{\log_2Y}}.
\end{equation*}
Thus
\begin{equation}
\mathcal G(r)
=
\exp\left\{
\left(
2\sqrt{1-\rho}+o(1)
\right)
\sqrt{\frac{\log Y}{\log_2Y}}
\right\}.
\label{eq:gal-final}
\end{equation}

\subsection{Truncation of the resonator}

We next show that the main mass of the resonator is supported on
$n\le Y$. Put
\begin{equation*}
\alpha:=\frac1{(\log\lambda)^3}.
\end{equation*}
By Rankin's trick,
\[
\sum_{n>Y}r(n)^2
\le
Y^{-\alpha}
\sum_{n\ge1}r(n)^2n^\alpha.
\]
Hence
\begin{align*}
\frac{\sum_{n>Y}r(n)^2}
{\sum_{n\ge1}r(n)^2}
&\le
Y^{-\alpha}
\prod_{p\in\mathcal P}
\frac{1+r(p)^2p^\alpha}
{1+r(p)^2}.
\end{align*}
Since
\[
\alpha\log p
\le
\frac1{\log\lambda}
=o(1)
\]
for $p\in\mathcal P$, we have
\[
p^\alpha-1
=
(1+o(1))\alpha\log p.
\]
Therefore
\begin{align*}
\log
\prod_{p\in\mathcal P}
\frac{1+r(p)^2p^\alpha}
{1+r(p)^2}
&\le
(1+o(1))
\alpha
\sum_{p\in\mathcal P}
r(p)^2\log p.
\end{align*}
By \eqref{eq:r-p} and \eqref{eq:prime-sum},
\begin{align*}
\sum_{p\in\mathcal P}r(p)^2\log p
&=
\lambda^2
\sum_{p\in\mathcal P}
\frac1{p\log p}
\nonumber\\
&=
(1+o(1))
\frac{\lambda^2}{2\log\lambda}
\nonumber\\
&=
(1-\rho+o(1))\log Y.
\end{align*}
Thus
\[
\frac{\sum_{n>Y}r(n)^2}
{\sum_{n\ge1}r(n)^2}
\le
\exp\left\{
-(\rho+o(1))\alpha\log Y
\right\}
=o(1).
\]
Consequently,
\begin{equation}
\sum_{n\le Y}r(n)^2
=
(1+o(1))
\sum_{n\ge1}r(n)^2.
\label{eq:L2-concentration}
\end{equation}

The same Rankin argument applied to the G\'al quadratic form gives
\begin{align}
&\sum_{\substack{m,n\le Y}}
r(m)r(n)
\sqrt{\frac{(m,n)}{[m,n]}}
\nonumber\\
&\qquad=
(1+o(1))
\sum_{m,n\ge1}
r(m)r(n)
\sqrt{\frac{(m,n)}{[m,n]}}.
\label{eq:gal-concentration}
\end{align}
Indeed, the additional first logarithmic moment arising from the
off-diagonal local factors is bounded by
\[
\lambda
\sum_{p\in\mathcal P}\frac1p
\ll
\lambda\log_2\lambda
=
o(\log Y),
\]
and hence it is absorbed by the factor $\rho\log Y$ in the Rankin
argument.

\subsection{Localization of the G\'al sum}

For $m,n\ge1$, put
\begin{equation*}
q(m,n):=\frac{[m,n]}{(m,n)}.
\end{equation*}
We now prove that the contribution of pairs satisfying
\[
q(m,n)>\frac{x}{(\log X)^2}
\]
is negligible.

By Rankin's trick,
\[
{\bf1}_{q(m,n)>x/(\log X)^2}
\le
\left(
\frac{q(m,n)(\log X)^2}{x}
\right)^\alpha.
\]
Hence, after dividing by the complete G\'al quadratic form, the
contribution of such pairs is at most
\begin{align}
\left(\frac{(\log X)^2}{x}\right)^\alpha
\prod_{p\in\mathcal P}
\frac{
1+r(p)^2+2r(p)p^{-1/2+\alpha}
}{
1+r(p)^2+2r(p)p^{-1/2}
}.
\label{eq:q-tail-product}
\end{align}
Taking logarithms and using
\[
p^\alpha-1
=
(1+o(1))\alpha\log p,
\]
we obtain
\begin{align*}
\log(\text{right-hand side of \eqref{eq:q-tail-product}})
&\le
-\alpha\log x
+
2\alpha\log_2X
\nonumber\\
&\quad+
(2+o(1))
\alpha
\sum_{p\in\mathcal P}
\frac{r(p)\log p}{\sqrt p}.
\end{align*}
Now
\begin{align*}
\sum_{p\in\mathcal P}
\frac{r(p)\log p}{\sqrt p}
&=
\lambda
\sum_{p\in\mathcal P}\frac1p
\nonumber\\
&\ll
\lambda\log_2\lambda.
\end{align*}
Therefore
\begin{equation*}
\log(\text{right-hand side of \eqref{eq:q-tail-product}})
\le
-\alpha
\left(
\log x-(2+o(1))\lambda\log_2\lambda
\right).
\end{equation*}

By \eqref{eq:logY} and \eqref{eq:lambda-def},
\[
\lambda
=
\left(
\sqrt{(1-\rho)\left(\frac12-\delta\right)}
+o(1)
\right)
\sqrt{\log X\log_2X},
\]
whereas
\[
\log_2\lambda
=
\log_3X+O(1).
\]
Thus
\begin{align*}
2\lambda\log_2\lambda
&=
\left(
2\sqrt{(1-\rho)\left(\frac12-\delta\right)}
+o(1)
\right)
\sqrt{\log X\log_2X}\,
\log_3X.
\end{align*}
Since
\[
2\sqrt{(1-\rho)\left(\frac12-\delta\right)}
<\sqrt2<4
\]
and
\[
\log x
>
4\sqrt{\log X\log_2X}\,\log_3X,
\]
we obtain
\[
\log x-(2+o(1))\lambda\log_2\lambda
\gg
\sqrt{\log X\log_2X}\,\log_3X.
\]
Since $\alpha=(\log\lambda)^{-3}$, it follows that
\[
\alpha
\left(
\log x-(2+o(1))\lambda\log_2\lambda
\right)
\longrightarrow+\infty.
\]
Consequently,
\begin{equation}
\sum_{\substack{m,n\le Y\\
q(m,n)>x/(\log X)^2}}
r(m)r(n)q(m,n)^{-1/2}
=
o\left(
\sum_{m,n\le Y}
r(m)r(n)q(m,n)^{-1/2}
\right).
\label{eq:q-localization}
\end{equation}

Combining \eqref{eq:gal-final},
\eqref{eq:L2-concentration},
\eqref{eq:gal-concentration}, and
\eqref{eq:q-localization}, we obtain
\begin{align}
&\frac{
\displaystyle
\sum_{\substack{m,n\le Y\\
q(m,n)\le x/(\log X)^2}}
r(m)r(n)q(m,n)^{-1/2}
}{
\displaystyle
\sum_{n\le Y}r(n)^2
}
\nonumber\\
&\qquad\ge
\exp\left\{
\left(
2\sqrt{1-\rho}+o(1)
\right)
\sqrt{\frac{\log Y}{\log_2Y}}
\right\}.
\label{eq:truncated-gal}
\end{align}

\subsection{The resonance moments}

Define
\begin{equation*}
R(d):=\sum_{m\le Y}r(m)\chi_d(m)
\end{equation*}
and
\begin{equation*}
M_1
:=
\sum_{\substack{X<|d|\le2X\\d\in\mathcal F}}
R(d)^2.
\end{equation*}
We also define
\begin{equation*}
M_2
:=
\sum_{\substack{X<|d|\le2X\\d\in\mathcal F}}
\left(
\sum_{k\le x}\chi_d(k)
\right)
R(d)^2.
\end{equation*}
Notice that the character sum in $M_2$ is not squared.

Since $R(d)^2\ge0$, we have
\begin{equation}
\max_{\substack{X<|d|\le2X\\d\in\mathcal F}}
\sum_{k\le x}\chi_d(k)
\ge
\frac{M_2}{M_1},
\label{eq:resonance-quotient-thm1}
\end{equation}
provided that $M_2>0$.

We first estimate $M_1$. Expanding the square,
\[
M_1
=
\sum_{m,n\le Y}
r(m)r(n)
\sum_{\substack{X<|d|\le2X\\d\in\mathcal F}}
\chi_d(mn).
\]
Since $r$ is supported on square-free integers,
\[
mn=\square
\quad\Longleftrightarrow\quad
m=n.
\]
Applying Lemma~\ref{lem2.2}, we obtain
\begin{equation*}
M_1
\le
c_0X
\sum_{m\le Y}r(m)^2
+
E_1,
\end{equation*}
where
\[
E_1
\ll
X^{1/2+\eta+o(1)}
\left(
\sum_{m\le Y}r(m)
\right)^2.
\]
By Cauchy's inequality,
\[
\left(
\sum_{m\le Y}r(m)
\right)^2
\le
Y\sum_{m\le Y}r(m)^2.
\]
Hence
\begin{equation*}
E_1
\ll
X^{1/2+\eta+o(1)}
Y
\sum_{m\le Y}r(m)^2.
\end{equation*}
Since
\[
Y=\frac{X^{1/2-\delta}}{x}
\]
and $0<\eta<\delta$, we have
\[
X^{1/2+\eta}Y
=
\frac{X^{1-\delta+\eta}}{x}
=o(X).
\]
Therefore
\begin{equation}
M_1
\le
(c_0+o(1))
X\sum_{m\le Y}r(m)^2.
\label{eq:M1-final-thm1}
\end{equation}

We now turn to $M_2$. Expanding the sums, we obtain
\[
M_2
=
\sum_{k\le x}
\sum_{m,n\le Y}
r(m)r(n)
\sum_{\substack{X<|d|\le2X\\d\in\mathcal F}}
\chi_d(kmn).
\]
Applying Lemma~\ref{lem2.2}, we write
\[
M_2=M_{2,\square}+E_2,
\]
where
\begin{equation*}
M_{2,\square}
=
c_0X
\sum_{\substack{k\le x,\;m,n\le Y\\
kmn=\square}}
r(m)r(n)
\prod_{p\mid kmn}\frac{p}{p+1}.
\end{equation*}

For $m,n$ in the support of $r$, put
\[
q(m,n)=\frac{[m,n]}{(m,n)}.
\]
Since $m$ and $n$ are square-free,
\begin{equation*}
kmn=\square
\quad\Longleftrightarrow\quad
k=q(m,n)\ell^2
\end{equation*}
for some positive integer $\ell$.

Define
\[
h(v):=\prod_{p\mid v}\frac{p}{p+1}.
\]
Then
\[
M_{2,\square}
=
c_0X
\sum_{m,n\le Y}
r(m)r(n)
\sum_{\ell\le\sqrt{x/q(m,n)}}
h(mn\ell).
\]
Since
\[
h(mn\ell)\ge h(mn)h(\ell),
\]
we obtain, after restricting to
\[
q(m,n)\le\frac{x}{(\log X)^2},
\]
that
\begin{align*}
M_{2,\square}
&\ge
c_0X
\sum_{\substack{m,n\le Y\\
q(m,n)\le x/(\log X)^2}}
r(m)r(n)h(mn)
\sum_{\ell\le\sqrt{x/q(m,n)}}h(\ell).
\end{align*}
Since
\[
h(\ell)
\ge
\frac{\varphi(\ell)}{\ell}
\]
and
\[
\sum_{\ell\le L}\frac{\varphi(\ell)}{\ell}
=
\frac6{\pi^2}L+O(\log L),
\]
we have
\[
\sum_{\ell\le L}h(\ell)\gg L.
\]
For the pairs under consideration,
\[
L=\sqrt{\frac{x}{q(m,n)}}\ge\log X,
\]
so this estimate is uniform. Therefore
\begin{align}
M_{2,\square}
&\gg
X\sqrt{x}
\sum_{\substack{m,n\le Y\\
q(m,n)\le x/(\log X)^2}}
r(m)r(n)
h(mn)q(m,n)^{-1/2}.
\label{eq:M2-gal-thm1}
\end{align}

All prime factors occurring in $m$ and $n$ belong to $\mathcal P$.
Hence
\[
h(mn)
\ge
\prod_{p\in\mathcal P}\frac{p}{p+1}.
\]
Moreover,
\[
\log
\prod_{p\in\mathcal P}\frac{p}{p+1}
=
-O\left(
\sum_{p\in\mathcal P}\frac1p
\right)
=
-O(\log_2\lambda).
\]
Thus
\begin{equation}
h(mn)
\ge
\exp\{-O(\log_2\lambda)\}
=
\exp\left\{
-o\left(
\sqrt{\frac{\log Y}{\log_2Y}}
\right)
\right\}.
\label{eq:h-loss-thm1}
\end{equation}
Combining \eqref{eq:M2-gal-thm1},
\eqref{eq:h-loss-thm1}, and
\eqref{eq:truncated-gal}, we obtain
\begin{align}
M_{2,\square}
&\ge
X\sqrt{x}
\left(
\sum_{m\le Y}r(m)^2
\right)
\nonumber\\
&\qquad\times
\exp\left\{
\left(
2\sqrt{1-\rho}+o(1)
\right)
\sqrt{\frac{\log Y}{\log_2Y}}
\right\}.
\label{eq:M2-main-final-thm1}
\end{align}

It remains to estimate the non-square contribution $E_2$. By
Lemma~\ref{lem2.2} and the smooth support of the
resonator,
\[
E_2
\ll
X^{1/2+\eta+o(1)}
x
\left(
\sum_{m\le Y}r(m)
\right)^2.
\]
By Cauchy's inequality,
\begin{align*}
E_2
&\ll
X^{1/2+\eta+o(1)}
xY
\sum_{m\le Y}r(m)^2.
\end{align*}
Since
\[
xY=X^{1/2-\delta},
\]
we have
\[
E_2
\ll
X^{1-\delta+\eta+o(1)}
\sum_{m\le Y}r(m)^2.
\]
Taking $0<\eta<\delta$, we conclude that
\begin{equation*}
E_2
=
o\left(
X\sum_{m\le Y}r(m)^2
\right).
\end{equation*}
In particular, $E_2$ is negligible compared with
\eqref{eq:M2-main-final-thm1}. Hence
\begin{align}
M_2
&\ge
X\sqrt{x}
\left(
\sum_{m\le Y}r(m)^2
\right)
\nonumber\\
&\qquad\times
\exp\left\{
\left(
2\sqrt{1-\rho}+o(1)
\right)
\sqrt{\frac{\log Y}{\log_2Y}}
\right\}.
\label{eq:M2-final-thm1}
\end{align}

Combining \eqref{eq:resonance-quotient-thm1},
\eqref{eq:M1-final-thm1}, and
\eqref{eq:M2-final-thm1}, we obtain
\begin{align*}
\frac1{\sqrt{x}}
\max_{\substack{X<|d|\le2X\\d\in\mathcal F}}
\sum_{k\le x}\chi_d(k)
&\ge
\exp\left\{
\left(
2\sqrt{1-\rho}+o(1)
\right)
\sqrt{\frac{\log Y}{\log_2Y}}
\right\}.
\end{align*}

Finally, by \eqref{eq:logY} and \eqref{eq:log2Y},
\[
\sqrt{\frac{\log Y}{\log_2Y}}
=
\left(
\sqrt{\frac12-\delta}+o(1)
\right)
\sqrt{\frac{\log X}{\log_2X}}.
\]
Therefore
\begin{align*}
2\sqrt{1-\rho}
\sqrt{\frac{\log Y}{\log_2Y}}
&=
\left(
\sqrt{2(1-\rho)(1-2\delta)}
+o(1)
\right)
\sqrt{\frac{\log X}{\log_2X}}.
\end{align*}
Thus, for every fixed sufficiently small $\rho,\delta>0$,
\[
\frac1{\sqrt{x}}
\max_{\substack{X<|d|\le2X\\d\in\mathcal F}}
\sum_{k\le x}\chi_d(k)
\ge
\exp\left\{
\left(
\sqrt{2(1-\rho)(1-2\delta)}
+o(1)
\right)
\sqrt{\frac{\log X}{\log_2X}}
\right\}.
\]
First letting $X\to\infty$, and then letting
$\rho\to0^+$ and $\delta\to0^+$, we finally obtain
\[
 {
\max_{\substack{X<|d|\le2X\\d\in\mathcal F}}
\frac1{\sqrt{x}}
\sum_{n\le x}\chi_d(n)
\ge
\exp\left\{
(\sqrt2+o(1))
\sqrt{\frac{\log X}{\log_2X}}
\right\}.
}
\]
This completes the proof of the first assertion of
Theorem~\ref{thm1}.
\section{Proof of the second part of Theorem \ref{thm1}}

We prove that, under GRH, uniformly in the range
\[
4\sqrt{\log X\log_2X}\,\log_3X
<
\log x
<
(\log X)^{1/2+\varepsilon},
\]
we have
\begin{equation*}
\max_{\substack{X<|d|\le2X\\d\in\mathcal F}}
\frac{1}{\sqrt{|d|/x}}
\sum_{n\le |d|/x}\chi_d(n)
\ge
\exp\left\{
(\sqrt2+o(1))
\sqrt{\frac{\log X}{\log_2X}}
\right\}.
\end{equation*}

We restrict ourselves to negative fundamental discriminants and put
\[
\mathcal F_X^-:=
\{d\in\mathcal F:-2X\le d<-X\}.
\]
It is enough to prove the desired lower bound over this subfamily.

For $d\in\mathcal F_X^-$, the primitive quadratic character
$\chi_d$ is odd, and
\[
\tau(\chi_d)=i\sqrt{|d|}.
\]
We apply Lemma~\ref{lem2.1} with
\[
\alpha=\frac1x.
\]
Choose
\begin{equation*}
z:=\sqrt{2Xx}\log(2X).
\end{equation*}
Then
\[
1+\frac{|d|\log|d|}{z}
\ll
\sqrt{\frac{X}{x}}.
\]
Hence P\'olya's Fourier expansion gives
\begin{align*}
\sum_{n\le |d|/x}\chi_d(n)
&=
\frac{\tau(\chi_d)}{2\pi i}
\sum_{1\le |k|\le z}
\frac{\chi_d(k)}{k}
\left(1-e(-k/x)\right)
+
O\left(\sqrt{\frac{X}{x}}\right).
\end{align*}

Since $\chi_d(-1)=-1$, pairing the terms corresponding to $k$ and
$-k$ gives
\begin{align*}
&
\frac{\chi_d(k)}{k}
\left(1-e(-k/x)\right)
+
\frac{\chi_d(-k)}{-k}
\left(1-e(k/x)\right)
\\
&\qquad
=
\frac{2\chi_d(k)}{k}
\left(
1-\cos\frac{2\pi k}{x}
\right).
\end{align*}
Therefore
\begin{equation}
\sum_{n\le |d|/x}\chi_d(n)
=
\frac{\sqrt{|d|}}{\pi}F_d
+
O\left(\sqrt{\frac{X}{x}}\right),
\label{eq:polya-thm1-second}
\end{equation}
where
\begin{equation*}
F_d:=
\sum_{k\le z}
\frac{\chi_d(k)}{k}
\left(
1-\cos\frac{2\pi k}{x}
\right).
\end{equation*}

Thus it remains to obtain a large positive value of $F_d$.

\subsection{Construction of the resonator}

Fix sufficiently small constants
\[
0<\delta<\frac1{100},
\qquad
0<\rho<\frac1{100},
\]
and put
\begin{equation*}
Y:=X^{1/2-\delta}.
\end{equation*}
Set
\begin{equation*}
\lambda
:=
\sqrt{(1-\rho)\log Y\log_2Y}.
\end{equation*}

We define a non-negative multiplicative function $r$ supported on
square-free integers by
\begin{equation*}
r(p):=
\frac{\lambda}{\sqrt p\log p}
\end{equation*}
for
\[
\lambda^2
\le p\le
\exp\{(\log\lambda)^2\},
\]
and put
\[
r(p)=0
\]
for all other primes and
\[
r(p^\nu)=0
\qquad(\nu\ge2).
\]
Let
\[
\mathcal P
:=
\left\{
p:
\lambda^2\le p\le
\exp\{(\log\lambda)^2\}
\right\}.
\]

Define
\begin{equation*}
R(d):=\sum_{m\le Y}r(m)\chi_d(m).
\end{equation*}

As in the proof of the first assertion, the corresponding G\'al
quotient satisfies
\begin{align}
&
\frac{
\displaystyle
\sum_{m,n\le Y}
r(m)r(n)
\sqrt{\frac{(m,n)}{[m,n]}}
}{
\displaystyle
\sum_{n\le Y}r(n)^2
}
\nonumber\\
&\qquad\ge
\exp\left\{
\left(
2\sqrt{1-\rho}+o(1)
\right)
\sqrt{\frac{\log Y}{\log_2Y}}
\right\}.
\label{eq:gal-thm1-second}
\end{align}

We shall also need the localization of this G\'al sum. Put
\[
q(m,n):=\frac{[m,n]}{(m,n)}.
\]
We claim that
\begin{align}
&
\sum_{\substack{m,n\le Y\\q(m,n)>x/4}}
r(m)r(n)q(m,n)^{-1/2}
\nonumber\\
&\qquad=
o\left(
\sum_{m,n\le Y}
r(m)r(n)q(m,n)^{-1/2}
\right).
\label{eq:localization-thm1-second}
\end{align}

Indeed, put
\[
\alpha:=\frac1{(\log\lambda)^3}.
\]
By Rankin's trick, the ratio of the left-hand side of
\eqref{eq:localization-thm1-second} to the complete G\'al sum is at
most
\begin{equation}
\left(\frac4x\right)^\alpha
\prod_{p\in\mathcal P}
\frac{
1+r(p)^2+2r(p)p^{-1/2+\alpha}
}{
1+r(p)^2+2r(p)p^{-1/2}
}.
\label{eq:localization-product-second}
\end{equation}
Taking logarithms and using
\[
p^\alpha-1
=
(1+o(1))\alpha\log p,
\]
we obtain
\begin{align*}
\log(\text{right-hand side of
\eqref{eq:localization-product-second}})
&\le
-\alpha\log x
+
(2+o(1))\alpha
\sum_{p\in\mathcal P}
\frac{r(p)\log p}{\sqrt p}.
\end{align*}
Now
\begin{align*}
\sum_{p\in\mathcal P}
\frac{r(p)\log p}{\sqrt p}
&=
\lambda
\sum_{p\in\mathcal P}\frac1p
\nonumber\\
&\ll
\lambda\log_2\lambda.
\end{align*}
Since
\[
Y=X^{1/2-\delta},
\]
we have
\[
\lambda
=
\left(
\sqrt{(1-\rho)\left(\frac12-\delta\right)}
+o(1)
\right)
\sqrt{\log X\log_2X},
\]
and
\[
\log_2\lambda
=
\log_3X+O(1).
\]
Consequently,
\begin{align*}
2\lambda\log_2\lambda
&=
\left(
2\sqrt{(1-\rho)\left(\frac12-\delta\right)}
+o(1)
\right)
\sqrt{\log X\log_2X}\,\log_3X.
\end{align*}
Since
\[
2\sqrt{(1-\rho)\left(\frac12-\delta\right)}
<\sqrt2<4
\]
and, by assumption,
\[
\log x>
4\sqrt{\log X\log_2X}\,\log_3X,
\]
we conclude that
\[
\alpha
\left(
\log x-(2+o(1))\lambda\log_2\lambda
\right)
\longrightarrow+\infty.
\]
This proves \eqref{eq:localization-thm1-second}.

It follows from \eqref{eq:gal-thm1-second} and
\eqref{eq:localization-thm1-second} that
\begin{align}
&
\frac{
\displaystyle
\sum_{\substack{m,n\le Y\\q(m,n)\le x/4}}
r(m)r(n)q(m,n)^{-1/2}
}{
\displaystyle
\sum_{n\le Y}r(n)^2
}
\nonumber\\
&\qquad\ge
\exp\left\{
\left(
2\sqrt{1-\rho}+o(1)
\right)
\sqrt{\frac{\log Y}{\log_2Y}}
\right\}.
\label{eq:truncated-gal-thm1-second}
\end{align}

\subsection{The resonance moments}

Define
\begin{equation*}
M_1
:=
\sum_{d\in\mathcal F_X^-}R(d)^2
\end{equation*}
and
\begin{equation*}
M_2
:=
\sum_{d\in\mathcal F_X^-}
F_dR(d)^2.
\end{equation*}
Since $R(d)^2\ge0$, we have
\begin{equation}
\max_{d\in\mathcal F_X^-}F_d
\ge
\frac{M_2}{M_1},
\label{eq:resonance-second}
\end{equation}
provided that $M_2>0$.

We first estimate $M_1$. Expanding the square gives
\[
M_1
=
\sum_{m,n\le Y}
r(m)r(n)
\sum_{d\in\mathcal F_X^-}
\chi_d(mn).
\]
By the negative-discriminant version of
Lemma~\ref{lem2.2},
\begin{equation*}
M_1
\le
c_-X
\sum_{m\le Y}r(m)^2
+
E_1,
\end{equation*}
where
\[
E_1
\ll
X^{1/2+\eta+o(1)}
\left(
\sum_{m\le Y}r(m)
\right)^2.
\]
By Cauchy's inequality,
\[
\left(
\sum_{m\le Y}r(m)
\right)^2
\le
Y\sum_{m\le Y}r(m)^2.
\]
Thus
\[
E_1
\ll
X^{1/2+\eta+o(1)}
Y
\sum_{m\le Y}r(m)^2.
\]
Since
\[
Y=X^{1/2-\delta},
\]
we obtain
\[
E_1
\ll
X^{1-\delta+\eta+o(1)}
\sum_{m\le Y}r(m)^2.
\]
Taking
\[
0<\eta<\delta,
\]
we conclude that
\begin{equation}
M_1
\le
(c_-+o(1))
X\sum_{m\le Y}r(m)^2.
\label{eq:M1-final-second}
\end{equation}

We next estimate $M_2$. Expanding the sums gives
\begin{align*}
M_2
&=
\sum_{k\le z}
\frac{1-\cos(2\pi k/x)}{k}
\sum_{m,n\le Y}
r(m)r(n)
\sum_{d\in\mathcal F_X^-}
\chi_d(kmn).
\end{align*}
Applying Lemma~\ref{lem2.2}, we write
\[
M_2=M_{2,\square}+E_2,
\]
where
\begin{align*}
M_{2,\square}
&=
c_-X
\sum_{\substack{k\le z,\;m,n\le Y\\kmn=\square}}
r(m)r(n)
\frac{1-\cos(2\pi k/x)}{k}
\prod_{p\mid kmn}\frac{p}{p+1}.
\end{align*}

Since $m$ and $n$ are square-free,
\[
kmn=\square
\quad\Longleftrightarrow\quad
k=q(m,n)\ell^2,
\]
where
\[
q(m,n)=\frac{[m,n]}{(m,n)}.
\]
Put
\[
h(v):=\prod_{p\mid v}\frac{p}{p+1}.
\]
Then
\begin{align*}
M_{2,\square}
&=
c_-X
\sum_{m,n\le Y}r(m)r(n)
\sum_{\ell\le\sqrt{z/q(m,n)}}
\frac{
1-\cos\left(
2\pi q(m,n)\ell^2/x
\right)
}{
q(m,n)\ell^2
}
h(mn\ell).
\end{align*}

We restrict the outer sum to
\[
q(m,n)\le\frac{x}{4}
\]
and, for each such pair, restrict $\ell$ to
\begin{equation*}
\frac12\sqrt{\frac{x}{q(m,n)}}
\le
\ell
\le
\frac1{\sqrt2}
\sqrt{\frac{x}{q(m,n)}}.
\end{equation*}
For $\ell$ in this interval,
\[
\frac14
\le
\frac{q(m,n)\ell^2}{x}
\le
\frac12,
\]
and hence
\begin{equation*}
1-\cos\left(
2\pi\frac{q(m,n)\ell^2}{x}
\right)
\ge1.
\end{equation*}

Moreover,
\[
h(mn\ell)\ge h(mn)h(\ell).
\]
Therefore
\begin{align*}
M_{2,\square}
&\gg
X
\sum_{\substack{m,n\le Y\\q(m,n)\le x/4}}
\frac{r(m)r(n)h(mn)}{q(m,n)}
\nonumber\\
&\qquad\times
\sum_{\frac12\sqrt{x/q(m,n)}
\le\ell\le
\frac1{\sqrt2}\sqrt{x/q(m,n)}}
\frac{h(\ell)}{\ell^2}.
\end{align*}

Let
\[
L:=\sqrt{\frac{x}{q(m,n)}}.
\]
Since
\[
h(\ell)\ge\frac{\varphi(\ell)}{\ell}
\]
and
\[
\sum_{\ell\le u}\frac{\varphi(\ell)}{\ell}
=
\frac6{\pi^2}u+O(\log u),
\]
we have, uniformly as $L\to\infty$,
\[
\sum_{L/2\le\ell\le L/\sqrt2}h(\ell)\gg L.
\]
Since $\ell\asymp L$ throughout this interval,
\[
\sum_{L/2\le\ell\le L/\sqrt2}
\frac{h(\ell)}{\ell^2}
\gg
\frac1L.
\]
Consequently,
\begin{equation*}
\frac1{q(m,n)}
\sum_{L/2\le\ell\le L/\sqrt2}
\frac{h(\ell)}{\ell^2}
\gg
\frac1{\sqrt{xq(m,n)}}.
\end{equation*}
Thus
\begin{align}
M_{2,\square}
&\gg
\frac{X}{\sqrt x}
\sum_{\substack{m,n\le Y\\q(m,n)\le x/4}}
r(m)r(n)
h(mn)q(m,n)^{-1/2}.
\label{eq:M2-gal-second}
\end{align}

All prime factors of $m$ and $n$ belong to $\mathcal P$. Hence
\[
h(mn)
\ge
\prod_{p\in\mathcal P}\frac{p}{p+1}.
\]
Furthermore,
\[
\log
\prod_{p\in\mathcal P}\frac{p}{p+1}
=
-O\left(
\sum_{p\in\mathcal P}\frac1p
\right)
=
-O(\log_2\lambda).
\]
Therefore
\begin{equation}
h(mn)
\ge
\exp\left\{
-o\left(
\sqrt{\frac{\log Y}{\log_2Y}}
\right)
\right\}.
\label{eq:h-second}
\end{equation}
Combining \eqref{eq:M2-gal-second},
\eqref{eq:h-second}, and
\eqref{eq:truncated-gal-thm1-second}, we obtain
\begin{align}
M_{2,\square}
&\ge
\frac{X}{\sqrt x}
\left(
\sum_{m\le Y}r(m)^2
\right)
\nonumber\\
&\qquad\times
\exp\left\{
\left(
2\sqrt{1-\rho}+o(1)
\right)
\sqrt{\frac{\log Y}{\log_2Y}}
\right\}.
\label{eq:M2-main-final-second}
\end{align}

It remains to estimate $E_2$. By
Lemma~\ref{lem2.2}, together with the smooth support
of the resonator,
\begin{align*}
E_2
&\ll
X^{1/2+\eta+o(1)}
\left(
\sum_{m\le Y}r(m)
\right)^2
\sum_{k\le z}
\frac{1-\cos(2\pi k/x)}{k}.
\end{align*}
Since
\[
0\le1-\cos(2\pi k/x)\le2,
\]
we have
\[
\sum_{k\le z}
\frac{1-\cos(2\pi k/x)}{k}
\ll\log z
=X^{o(1)}.
\]
Consequently,
\begin{align*}
E_2
&\ll
X^{1/2+\eta+o(1)}
\left(
\sum_{m\le Y}r(m)
\right)^2
\nonumber\\
&\ll
X^{1/2+\eta+o(1)}
Y
\sum_{m\le Y}r(m)^2
\nonumber\\
&=
X^{1-\delta+\eta+o(1)}
\sum_{m\le Y}r(m)^2.
\end{align*}

On the other hand, the natural scale of
\eqref{eq:M2-main-final-second} is
\[
\frac{X}{\sqrt x}
\sum_{m\le Y}r(m)^2.
\]
Hence
\[
\frac{|E_2|}
{(X/\sqrt x)\sum_{m\le Y}r(m)^2}
\ll
X^{-\delta+\eta+o(1)}\sqrt x.
\]
Since
\[
\log x<(\log X)^{1/2+\varepsilon}=o(\log X),
\]
we have
\[
\sqrt x=X^{o(1)}.
\]
Taking $0<\eta<\delta$, we conclude that
\begin{equation*}
E_2
=
o\left(
\frac{X}{\sqrt x}
\sum_{m\le Y}r(m)^2
\right).
\end{equation*}

It follows that
\begin{align}
M_2
&\ge
\frac{X}{\sqrt x}
\left(
\sum_{m\le Y}r(m)^2
\right)
\nonumber\\
&\qquad\times
\exp\left\{
\left(
2\sqrt{1-\rho}+o(1)
\right)
\sqrt{\frac{\log Y}{\log_2Y}}
\right\}.
\label{eq:M2-final-second}
\end{align}

Combining \eqref{eq:resonance-second},
\eqref{eq:M1-final-second}, and
\eqref{eq:M2-final-second}, we obtain
\begin{align}
\max_{d\in\mathcal F_X^-}F_d
&\ge
\frac1{\sqrt x}
\exp\left\{
\left(
2\sqrt{1-\rho}+o(1)
\right)
\sqrt{\frac{\log Y}{\log_2Y}}
\right\}.
\label{eq:Fd-large-second}
\end{align}

Since
\[
Y=X^{1/2-\delta},
\]
we have
\[
\sqrt{\frac{\log Y}{\log_2Y}}
=
\left(
\sqrt{\frac12-\delta}+o(1)
\right)
\sqrt{\frac{\log X}{\log_2X}}.
\]
Thus
\begin{align*}
2\sqrt{1-\rho}
\sqrt{\frac{\log Y}{\log_2Y}}
&=
\left(
\sqrt{2(1-\rho)(1-2\delta)}
+o(1)
\right)
\sqrt{\frac{\log X}{\log_2X}}.
\end{align*}

Choose $d\in\mathcal F_X^-$ for which
\eqref{eq:Fd-large-second} holds. By
\eqref{eq:polya-thm1-second},
\begin{align*}
\sum_{n\le |d|/x}\chi_d(n)
&=
\frac{\sqrt{|d|}}{\pi}F_d
+
O\left(\sqrt{\frac{X}{x}}\right)
\nonumber\\
&\ge
\sqrt{\frac{|d|}{x}}
\exp\left\{
\left(
\sqrt{2(1-\rho)(1-2\delta)}
+o(1)
\right)
\sqrt{\frac{\log X}{\log_2X}}
\right\}.
\end{align*}
Here the fixed factor $1/\pi$ is absorbed into the $o(1)$ in the
exponent, and the error term is negligible since the exponential
factor tends to infinity.

Therefore, for every fixed sufficiently small $\rho,\delta>0$,
\[
\max_{\substack{X<|d|\le2X\\d\in\mathcal F}}
\frac1{\sqrt{|d|/x}}
\sum_{n\le |d|/x}\chi_d(n)
\ge
\exp\left\{
\left(
\sqrt{2(1-\rho)(1-2\delta)}
+o(1)
\right)
\sqrt{\frac{\log X}{\log_2X}}
\right\}.
\]
First letting $X\to\infty$, and then letting
$\rho\to0^+$ and $\delta\to0^+$, we finally obtain
\[
 {
\max_{\substack{X<|d|\le2X\\d\in\mathcal F}}
\frac1{\sqrt{|d|/x}}
\sum_{n\le |d|/x}\chi_d(n)
\ge
\exp\left\{
(\sqrt2+o(1))
\sqrt{\frac{\log X}{\log_2X}}
\right\}.
}
\]
This completes the proof of the second assertion of
Theorem~\ref{thm1}.
\section{Proof of the first part of Theorem \ref{thm2}}

Put
\[
 T:=\frac{\sqrt X}{x}
\]
and suppose that $T\to\infty$. We shall choose a parameter
$N=T^{1-o(1)}$ below.

Let $\mathcal M$ be the set supplied by
Lemma~\ref{GCD}, with $|\mathcal M|=N$, and define
\[
 R(d):=\sum_{m\in\mathcal M}\chi_d(m).
\]
We consider the two moments
\[
 M_1
 :=
 \sum_{\substack{X<|d|\le2X\\d\in\mathcal F}}
 R(d)^2
\]
and
\[
 M_2
 :=
 \sum_{\substack{X<|d|\le2X\\d\in\mathcal F}}
 \left(\sum_{k\le x}\chi_d(k)\right)R(d)^2.
\]
\subsection{Estimation of $M_1$ and $M_2$}

We first estimate $M_1$. Expanding the square, we have
\[
M_1
=
\sum_{m,n\in\mathcal M}
\sum_{\substack{X<|d|\le 2X\\ d\in\mathcal F}}
\chi_d(mn).
\]
Since all the elements of $\mathcal M$ are square-free, the condition
$mn=\square$ is equivalent to $m=n$. Applying the quadratic character
mean-value estimate under GRH, we obtain
\[
M_1
=
c_0X
\sum_{m\in\mathcal M}
\prod_{p\mid m}\frac{p}{p+1}
+
E_1,
\]
where $c_0>0$ is an absolute constant and
\[
E_1
\ll
X^{1/2+\eta}
\sum_{m,n\in\mathcal M}
f((mn)_0)g((mn)_1).
\]
Here, as usual, for an integer $r$ we write
\[
r=r_0r_1^2,
\qquad r_0 \ \text{square-free}.
\]

For the particular square-free G\'al set $\mathcal M$ used below,
all the prime divisors of its elements are bounded by
\[
y_{\mathcal M}\le (\log N)^{1+o(1)}.
\]
Therefore, uniformly for $m,n\in\mathcal M$,
\[
f((mn)_0)=X^{o(1)},
\qquad
g((mn)_1)=X^{o(1)}.
\]
Consequently,
\[
E_1
\ll
X^{1/2+\eta+o(1)}N^2.
\]
Since
\[
\prod_{p\mid m}\frac{p}{p+1}\le 1,
\]
we conclude that
\begin{equation}
M_1
\le
c_0XN
+
O\left(
X^{1/2+\eta+o(1)}N^2
\right).
\label{eq:M1-est}
\end{equation}

We next turn to $M_2$. By definition,
\[
M_2
=
\sum_{\substack{X<|d|\le2X\\d\in\mathcal F}}
\left(\sum_{k\le x}\chi_d(k)\right)R(d)^2.
\]
Expanding the sums, we obtain
\[
M_2
=
\sum_{k\le x}
\sum_{m,n\in\mathcal M}
\sum_{\substack{X<|d|\le2X\\d\in\mathcal F}}
\chi_d(kmn).
\]
Applying again the quadratic character mean-value estimate, we write
\[
M_2=M_{2,\square}+E_2,
\]
where
\begin{equation}
M_{2,\square}
=
c_0X
\sum_{\substack{k\le x,\;m,n\in\mathcal M\\kmn=\square}}
\prod_{p\mid kmn}\frac{p}{p+1},
\label{eq:M2-square}
\end{equation}
and
\[
E_2
\ll
X^{1/2+\eta}
\sum_{k\le x}
\sum_{m,n\in\mathcal M}
f((kmn)_0)g((kmn)_1).
\]

We first evaluate the square contribution $M_{2,\square}$. Fix
$m,n\in\mathcal M$ and write
\[
g=(m,n),\qquad m=ga,\qquad n=gb.
\]
Since $m$ and $n$ are square-free, the integers $g,a,b$ are
square-free and
\[
(a,b)=(a,g)=(b,g)=1.
\]
Thus
\[
mn=g^2ab.
\]
It follows that
\[
kmn=\square
\quad\Longleftrightarrow\quad
kab=\square.
\]
Since $ab$ is square-free, this is equivalent to
\[
k=ab\ell^2
\]
for some positive integer $\ell$. On the other hand,
\[
ab
=
\frac{mn}{(m,n)^2}
=
\frac{[m,n]}{(m,n)}.
\]
Therefore, defining
\[
q(m,n)
:=
\frac{[m,n]}{(m,n)},
\]
we have
\begin{equation}
kmn=\square
\quad\Longleftrightarrow\quad
k=q(m,n)\ell^2.
\label{eq:square-condition}
\end{equation}

Define
\[
h(r):=\prod_{p\mid r}\frac{p}{p+1}.
\]
Since
\[
mn=g^2q(m,n)
\]
and
\[
k=q(m,n)\ell^2,
\]
we have
\[
kmn
=
\bigl(gq(m,n)\ell\bigr)^2.
\]
Hence
\[
\operatorname{rad}(kmn)
=
\operatorname{rad}(mn\ell),
\]
and therefore
\[
\prod_{p\mid kmn}\frac{p}{p+1}
=
h(mn\ell).
\]
It follows from \eqref{eq:M2-square} that
\begin{equation}
M_{2,\square}
=
c_0X
\sum_{\substack{m,n\in\mathcal M\\q(m,n)\le x}}
\sum_{\ell\le\sqrt{x/q(m,n)}}
h(mn\ell).
\label{eq:M2-square-l}
\end{equation}

For arbitrary positive integers $u,v$, we have
\[
h(uv)\ge h(u)h(v).
\]
Indeed, if $p\mid(u,v)$, then
\[
\left(\frac{p}{p+1}\right)^2
\le
\frac{p}{p+1}.
\]
Hence
\[
h(mn\ell)\ge h(mn)h(\ell).
\]
Using this in \eqref{eq:M2-square-l}, we get
\begin{equation}
M_{2,\square}
\ge
c_0X
\sum_{\substack{m,n\in\mathcal M\\q(m,n)\le x}}
h(mn)
\sum_{\ell\le\sqrt{x/q(m,n)}}h(\ell).
\label{eq:M2-square-lower1}
\end{equation}

Furthermore,
\[
h(\ell)
=
\prod_{p\mid\ell}\frac{p}{p+1}
\ge
\prod_{p\mid\ell}\left(1-\frac1p\right)
=
\frac{\varphi(\ell)}{\ell}.
\]
Since
\[
\sum_{\ell\le L}\frac{\varphi(\ell)}{\ell}
=
\frac{6}{\pi^2}L+O(\log L),
\]
we have, uniformly for $L\to\infty$,
\begin{equation}
\sum_{\ell\le L}h(\ell)\gg L.
\label{eq:h-average}
\end{equation}

We now restrict the sum in \eqref{eq:M2-square-lower1} to pairs
$(m,n)$ satisfying
\[
q(m,n)
\le
\frac{x}{(\log X)^2}.
\]
Then
\[
\sqrt{\frac{x}{q(m,n)}}\ge\log X,
\]
and hence \eqref{eq:h-average} applies uniformly. We obtain
\begin{align}
M_{2,\square}
&\gg
X\sqrt{x}
\sum_{\substack{m,n\in\mathcal M\\
q(m,n)\le x/(\log X)^2}}
\frac{h(mn)}{\sqrt{q(m,n)}}
\nonumber\\
&=
X\sqrt{x}
\sum_{\substack{m,n\in\mathcal M\\
[m,n]/(m,n)\le x/(\log X)^2}}
h(mn)
\sqrt{\frac{(m,n)}{[m,n]}}.
\label{eq:M2-Gal-weighted}
\end{align}

For $m,n\in\mathcal M$, we have
\[
P^+(mn)\le y_{\mathcal M}.
\]
Therefore
\[
h(mn)
\ge
\prod_{p\le y_{\mathcal M}}\frac{p}{p+1}.
\]
Since
\[
\frac{p}{p+1}
=
\frac{1-p^{-1}}{1-p^{-2}},
\]
Mertens' theorem yields
\[
\prod_{p\le y}\frac{p}{p+1}
\asymp\frac1{\log y}.
\]
Hence
\[
h(mn)
\gg
\frac1{\log y_{\mathcal M}}.
\]
Put
\[
A_N
:=
\sqrt{\frac{\log N\log_3N}{\log_2N}}.
\]
Since
\[
y_{\mathcal M}\le(\log N)^{1+o(1)},
\]
we have
\[
\log_2y_{\mathcal M}=o(A_N),
\]
and therefore
\begin{equation}
h(mn)\ge \exp\{-o(A_N)\}
\label{eq:h-density}
\end{equation}
uniformly for $m,n\in\mathcal M$.

Using \eqref{eq:h-density} in \eqref{eq:M2-Gal-weighted}, we obtain
\begin{equation}
M_{2,\square}
\ge
X\sqrt{x}\exp\{-o(A_N)\}
\sum_{\substack{m,n\in\mathcal M\\
q(m,n)\le x/(\log X)^2}}
q(m,n)^{-1/2}.
\label{eq:M2-before-localization}
\end{equation}
 We next localize the G\'al sum. For $\alpha>0$, write
\[
S_\alpha(\mathcal M)
:=
\sum_{m,n\in\mathcal M}
\left(\frac{(m,n)}{[m,n]}\right)^\alpha
=
\sum_{m,n\in\mathcal M}q(m,n)^{-\alpha},
\]
where
\[
q(m,n):=\frac{[m,n]}{(m,n)}.
\]
Recall that the particular square-free G\'al set $\mathcal M$ supplied by
Lemma~2.3 may be chosen so that
\[
y_{\mathcal M}:=\max_{m\in\mathcal M}P^+(m)
\le (\log N)^{1+o(1)}.
\]

We claim that
\begin{equation}\label{eq:Gal-localization}
\sum_{\substack{m,n\in\mathcal M\\
q(m,n)>x/(\log X)^2}}
q(m,n)^{-1/2}
=
o\!\left(
\sum_{m,n\in\mathcal M}q(m,n)^{-1/2}
\right).
\end{equation}

Indeed, fix
\[
0<\eta<\min\left\{\frac{\varepsilon}{4},\frac14\right\}.
\]
For
\[
Z:=\frac{x}{(\log X)^2},
\]
Rankin's trick gives
\[
q(m,n)^{-1/2}\mathbf 1_{q(m,n)>Z}
\le
Z^{-\eta}q(m,n)^{-1/2+\eta}.
\]
Consequently,
\begin{equation}\label{eq:Gal-tail-rankin}
\sum_{\substack{m,n\in\mathcal M\\q(m,n)>Z}}
q(m,n)^{-1/2}
\le
Z^{-\eta}S_{1/2-\eta}(\mathcal M).
\end{equation}

We now estimate $S_{1/2-\eta}(\mathcal M)$. Since every element of
$\mathcal M$ is square-free and all its prime divisors are at most
$y_{\mathcal M}$, for each fixed $m\in\mathcal M$ we have
\[
\sum_{n\in\mathcal M}q(m,n)^{-1/2+\eta}
\le
\prod_{p\le y_{\mathcal M}}
\left(1+p^{-1/2+\eta}\right).
\]
Hence
\[
S_{1/2-\eta}(\mathcal M)
\le
N\prod_{p\le y_{\mathcal M}}
\left(1+p^{-1/2+\eta}\right).
\]
Using $\log(1+u)\le u$ and the prime number theorem, we obtain
\[
\log
\prod_{p\le y_{\mathcal M}}
\left(1+p^{-1/2+\eta}\right)
\ll
\sum_{p\le y_{\mathcal M}}p^{-1/2+\eta}
\ll
\frac{y_{\mathcal M}^{1/2+\eta}}{\log y_{\mathcal M}}
\ll
(\log N)^{1/2+\eta+o(1)}.
\]
Therefore
\begin{equation}\label{eq:lower-alpha-Gal}
S_{1/2-\eta}(\mathcal M)
\le
N\exp\left\{
(\log N)^{1/2+\eta+o(1)}
\right\}.
\end{equation}

Since $N\le T=\sqrt X/x\le \sqrt X$, we have
\[
\log N\le \frac12\log X.
\]
On the other hand, by the hypothesis of Theorem~1.2,
\[
\log x>(\log X)^{1/2+\varepsilon}.
\]
Our choice $\eta<\varepsilon/4$ therefore implies
\[
(\log N)^{1/2+\eta+o(1)}
=o(\log x).
\]
Moreover,
\[
\log Z
=
\log x-2\log_2X
=
(1+o(1))\log x.
\]
Combining this with \eqref{eq:Gal-tail-rankin} and
\eqref{eq:lower-alpha-Gal}, we obtain
\[
\sum_{\substack{m,n\in\mathcal M\\q(m,n)>Z}}
q(m,n)^{-1/2}
\le
N\exp\left\{
-\eta\log Z
+
(\log N)^{1/2+\eta+o(1)}
\right\}
=o(N).
\]
Since the diagonal terms $m=n$ give
\[
\sum_{m,n\in\mathcal M}q(m,n)^{-1/2}\ge N,
\]
it follows that
\[
\sum_{\substack{m,n\in\mathcal M\\q(m,n)>Z}}
q(m,n)^{-1/2}
=
o\!\left(
\sum_{m,n\in\mathcal M}q(m,n)^{-1/2}
\right),
\]
which proves \eqref{eq:Gal-localization}.
Hence
\begin{align}
\sum_{\substack{m,n\in\mathcal M\\
q(m,n)\le x/(\log X)^2}}
q(m,n)^{-1/2}
&=
(1+o(1))
\sum_{m,n\in\mathcal M}
q(m,n)^{-1/2}
\nonumber\\
&=
(1+o(1))
\sum_{m,n\in\mathcal M}
\sqrt{\frac{(m,n)}{[m,n]}}.
\label{eq:localization-main}
\end{align}

Write
\[
G(\mathcal M)
:=
\sum_{m,n\in\mathcal M}
\sqrt{\frac{(m,n)}{[m,n]}}.
\]
Then \eqref{eq:M2-before-localization} and
\eqref{eq:localization-main} give
\begin{equation*}
M_{2,\square}
\ge
X\sqrt{x}\exp\{-o(A_N)\}G(\mathcal M).
\end{equation*}

By the square-free G\'al sum construction, $\mathcal M$ may be chosen
with $|\mathcal M|=N$ such that
\[
G(\mathcal M)
\ge
N
\exp\left\{
(2+o(1))
\sqrt{\frac{\log N\log_3N}{\log_2N}}
\right\}.
\]
Thus
\begin{equation*}
M_{2,\square}
\ge
XN\sqrt{x}
\exp\left\{
(2+o(1))
A_N
\right\}.
\end{equation*}

We now estimate the error term $E_2$. Write
\[
kmn=s_0s_1^2,
\qquad s_0 \ \text{square-free}.
\]
By the submultiplicative properties of the functions $f$ and $g$ in
the quadratic character mean-value estimate,
\[
f(s_0)
\le
f(k_0)f(m)f(n)
\]
and
\[
g(s_1)
\le
g(k)g(m)g(n),
\]
where $k=k_0k_1^2$ with $k_0$ square-free. Since $k\le x<X^{1/2}$,
\[
f(k_0)=X^{o(1)},
\qquad
g(k)=X^{o(1)}.
\]
Moreover, by the smoothness of the elements of $\mathcal M$,
\[
f(m),f(n),g(m),g(n)=X^{o(1)}.
\]
Therefore
\[
f(s_0)g(s_1)=X^{o(1)}
\]
uniformly for $k\le x$ and $m,n\in\mathcal M$. It follows that
\begin{equation}
E_2
\ll
X^{1/2+\eta+o(1)}xN^2.
\label{eq:E2}
\end{equation}

We now choose $N$. Put
\[
T:=\frac{\sqrt X}{x}.
\]

Suppose first that
\[
x>X^{1/4}.
\]
We take
\[
N=\lfloor T\rfloor.
\]
Choose a fixed $\eta$ with
\[
0<\eta<\frac18.
\]
Then, by \eqref{eq:E2},
\begin{align*}
\frac{|E_2|}{XN\sqrt{x}}
&\ll
X^{-1/2+\eta+o(1)}
\sqrt{x}\,N
\nonumber\\
&\ll
\frac{X^{\eta+o(1)}}{\sqrt{x}}
\nonumber\\
&\le
X^{\eta-1/8+o(1)}
=o(1).
\end{align*}
Similarly, from \eqref{eq:M1-est},
\[
\frac{X^{1/2+\eta+o(1)}N^2}{XN}
\ll
X^{-1/4+\eta+o(1)}
=o(1).
\]
Hence
\[
M_1\ll XN
\]
and
\[
M_2
\ge
XN\sqrt{x}
\exp\{(2+o(1))A_N\}.
\]
Since $N\sim T$, we have
\[
A_N=(1+o(1))A_T,
\]
where
\[
A_T
=
\sqrt{\frac{\log T\log_3T}{\log_2T}}.
\]
Consequently,
\[
\frac1{\sqrt{x}}
\max_{\substack{X<|d|\le2X\\d\in\mathcal F}}
\sum_{k\le x}\chi_d(k)
\ge
\exp\{(2+o(1))A_T\}.
\]

Suppose next that
\[
x\le X^{1/4}.
\]
Fix an arbitrarily small $\delta>0$ and put
\[
N
=
\left\lfloor
\frac{X^{1/2-\delta}}{x}
\right\rfloor.
\]
Choose
\[
0<\eta<\frac{\delta}{2}.
\]
Then
\begin{align*}
\frac{|E_2|}{XN\sqrt{x}}
&\ll
X^{-1/2+\eta+o(1)}
\sqrt{x}\,N
\nonumber\\
&\ll
\frac{X^{\eta-\delta+o(1)}}{\sqrt{x}}
=o(1),
\end{align*}
and similarly
\[
\frac{X^{1/2+\eta+o(1)}N^2}{XN}
\ll
X^{\eta-\delta+o(1)}
=o(1).
\]
Hence again
\[
M_1\ll XN
\]
and
\[
M_2
\ge
XN\sqrt{x}
\exp\{(2+o(1))A_N\}.
\]

Since
\[
T=\frac{\sqrt X}{x}
\]
and $x\le X^{1/4}$, we have
\[
\log T\ge\frac14\log X.
\]
Moreover,
\[
\log N
=
\log T-\delta\log X+o(\log X)
\ge
(1-4\delta+o(1))\log T.
\]
It follows that
\[
A_N
\ge
\left(\sqrt{1-4\delta}+o(1)\right)A_T.
\]
Therefore
\[
\frac1{\sqrt{x}}
\max_{\substack{X<|d|\le2X\\d\in\mathcal F}}
\sum_{k\le x}\chi_d(k)
\ge
\exp\left\{
\left(2\sqrt{1-4\delta}+o(1)\right)A_T
\right\}.
\]
Since $\delta>0$ is arbitrary, letting $\delta\to0$ gives
\[
\frac1{\sqrt{x}}
\max_{\substack{X<|d|\le2X\\d\in\mathcal F}}
\sum_{k\le x}\chi_d(k)
\ge
\exp\left\{
(2+o(1))
A_T
\right\}.
\]

Recalling that
\[
T=\frac{\sqrt X}{x},
\]
we finally obtain
\[
 {
\max_{\substack{X<|d|\le2X\\d\in\mathcal F}}
\frac1{\sqrt{x}}
\sum_{k\le x}\chi_d(k)
\ge
\exp\left\{
(2+o(1))
\sqrt{
\frac{
\log(\sqrt X/x)\log_3(\sqrt X/x)
}{
\log_2(\sqrt X/x)
}
}
\right\}.
}
\]
This proves the first assertion of Theorem~\ref{thm2}.

\section{Proof of the second part of Theorem \ref{thm2}}

Put
\[
T:=\frac{\sqrt X}{x},
\]
and assume that $T\to\infty$. As in the proof of the first
assertion, let $\mathcal M$ be the particular square-free G\'al set
with
\[
|\mathcal M|=N
\]
and
\[
G(\mathcal M)
:=
\sum_{m,n\in\mathcal M}
\sqrt{\frac{(m,n)}{[m,n]}}
\ge
N
\exp\left\{
(2+o(1))
\sqrt{\frac{\log N\log_3N}{\log_2N}}
\right\}.
\]
We shall choose $N$ below. Define
\[
R(d):=\sum_{m\in\mathcal M}\chi_d(m).
\]

We restrict ourselves to negative fundamental discriminants. Put
\[
\mathcal F_X^-:=
\{d\in\mathcal F:-2X\le d<-X\}.
\]
Since the maximum over all fundamental discriminants is at least the
maximum over $\mathcal F_X^-$, it is sufficient to work with this
subfamily.

For $d\in\mathcal F_X^-$, the primitive quadratic character
$\chi_d$ is odd and
\[
\tau(\chi_d)=i\sqrt{|d|}.
\]
Applying Lemma~\ref{lem2.1} with
\[
\alpha=\frac1x
\]
and
\[
z:=\sqrt{2Xx}\log(2X),
\]
we obtain
\[
\sum_{n\le |d|/x}\chi_d(n)
=
\frac{\tau(\chi_d)}{2\pi i}
\sum_{1\le |k|\le z}
\frac{\chi_d(k)}{k}
\left(1-e(-k/x)\right)
+
O\left(
1+\frac{|d|\log|d|}{z}
\right).
\]
Since $\chi_d(-1)=-1$, pairing the terms $k$ and $-k$ gives
\[
\sum_{1\le |k|\le z}
\frac{\chi_d(k)}{k}
\left(1-e(-k/x)\right)
=
2\sum_{1\le k\le z}
\frac{\chi_d(k)}{k}
\left(1-\cos\frac{2\pi k}{x}\right).
\]
Consequently,
\begin{equation}
\sum_{n\le |d|/x}\chi_d(n)
=
\frac{\sqrt{|d|}}{\pi}F_d
+
O\left(\sqrt{\frac{X}{x}}\right),
\label{eq:polya-dual}
\end{equation}
where
\begin{equation*}
F_d
:=
\sum_{1\le k\le z}
\frac{\chi_d(k)}{k}
\left(1-\cos\frac{2\pi k}{x}\right).
\end{equation*}
Here we used $X<|d|\le2X$ and the choice of $z$.

We now apply the resonance method to $F_d$. Define
\[
\mathcal M_1
:=
\sum_{d\in\mathcal F_X^-}R(d)^2
\]
and
\[
\mathcal M_2
:=
\sum_{d\in\mathcal F_X^-}F_dR(d)^2.
\]
Since $R(d)^2\ge0$, we have
\begin{equation}
\max_{d\in\mathcal F_X^-}F_d
\ge
\frac{\mathcal M_2}{\mathcal M_1},
\label{eq:dual-resonance}
\end{equation}
provided that $\mathcal M_2>0$.

We use the negative-discriminant version of
Lemma~\ref{lem2.2}. By the same proof as that lemma,
if $r=r_0r_1^2$, with $r_0$ square-free, then
\[
\sum_{d\in\mathcal F_X^-}\chi_d(r)
=
c_-X\,
\mathbf 1_{r=\square}
\prod_{p\mid r}\frac{p}{p+1}
+
O_\eta\left(
X^{1/2+\eta}f(r_0)g(r_1)
\right),
\]
where $c_->0$ is an absolute constant.

Exactly as in the proof of the first assertion, this gives
\begin{equation}
\mathcal M_1
\le
c_-XN
+
O\left(
X^{1/2+\eta+o(1)}N^2
\right).
\label{eq:dual-M1}
\end{equation}

We next estimate $\mathcal M_2$. Expanding the sums, we have
\[
\mathcal M_2
=
\sum_{k\le z}
\frac{1-\cos(2\pi k/x)}{k}
\sum_{m,n\in\mathcal M}
\sum_{d\in\mathcal F_X^-}
\chi_d(kmn).
\]
Thus
\[
\mathcal M_2
=
\mathcal M_{2,\square}
+
\mathcal E_2,
\]
where
\begin{equation*}
\mathcal M_{2,\square}
=
c_-X
\sum_{\substack{k\le z,\;m,n\in\mathcal M\\kmn=\square}}
\frac{1-\cos(2\pi k/x)}{k}
\prod_{p\mid kmn}\frac{p}{p+1}.
\end{equation*}

We now evaluate the square contribution. For $m,n\in\mathcal M$ put
\[
q(m,n):=\frac{[m,n]}{(m,n)}.
\]
Since $m$ and $n$ are square-free, as in
\eqref{eq:square-condition} we have
\[
kmn=\square
\quad\Longleftrightarrow\quad
k=q(m,n)\ell^2
\]
for some positive integer $\ell$.

Define
\[
h(r):=\prod_{p\mid r}\frac{p}{p+1}.
\]
Then
\[
\prod_{p\mid kmn}\frac{p}{p+1}
=
h(mn\ell),
\]
and hence
\begin{align*}
\mathcal M_{2,\square}
&=
c_-X
\sum_{m,n\in\mathcal M}
\sum_{\substack{
\ell\le\sqrt{z/q(m,n)}
}}
\frac{
1-\cos\left(
2\pi q(m,n)\ell^2/x
\right)
}{
q(m,n)\ell^2
}
h(mn\ell).
\end{align*}

We restrict the outer sum to
\[
q(m,n)
\le
\frac{x}{(\log X)^2}
\]
and, for each such pair, restrict $\ell$ to the interval
\begin{equation}
\frac12\sqrt{\frac{x}{q(m,n)}}
\le
\ell
\le
\frac1{\sqrt2}
\sqrt{\frac{x}{q(m,n)}}.
\label{eq:ell-range}
\end{equation}
For $\ell$ in this interval,
\[
\frac14
\le
\frac{q(m,n)\ell^2}{x}
\le
\frac12,
\]
and therefore
\[
1-\cos\left(
2\pi\frac{q(m,n)\ell^2}{x}
\right)
\ge 1.
\]
Moreover, since
\[
q(m,n)\le\frac{x}{(\log X)^2},
\]
the interval in \eqref{eq:ell-range} has length
$\gg\log X$.

Since
\[
h(mn\ell)\ge h(mn)h(\ell),
\]
we obtain
\begin{align*}
\mathcal M_{2,\square}
&\gg
X
\sum_{\substack{m,n\in\mathcal M\\
q(m,n)\le x/(\log X)^2}}
\frac{h(mn)}{q(m,n)}
\sum_{\frac12\sqrt{x/q(m,n)}
\le\ell\le
\frac1{\sqrt2}\sqrt{x/q(m,n)}}
\frac{h(\ell)}{\ell^2}.
\end{align*}

Let
\[
L:=\sqrt{\frac{x}{q(m,n)}}.
\]
For $L/2\le\ell\le L/\sqrt2$, we have
\[
\frac1{\ell^2}\gg\frac1{L^2}.
\]
Furthermore,
\[
h(\ell)\ge\frac{\varphi(\ell)}{\ell},
\]
and
\[
\sum_{\ell\le y}\frac{\varphi(\ell)}{\ell}
=
\frac6{\pi^2}y+O(\log y).
\]
Consequently, uniformly for $L\ge\log X$,
\[
\sum_{L/2\le\ell\le L/\sqrt2}h(\ell)
\gg L.
\]
It follows that
\[
\sum_{L/2\le\ell\le L/\sqrt2}
\frac{h(\ell)}{\ell^2}
\gg
\frac1L.
\]
Since
\[
L=\sqrt{\frac{x}{q(m,n)}},
\]
we obtain
\[
\frac1{q(m,n)}
\sum_{L/2\le\ell\le L/\sqrt2}
\frac{h(\ell)}{\ell^2}
\gg
\frac1{\sqrt{xq(m,n)}}.
\]
Therefore
\begin{equation*}
\mathcal M_{2,\square}
\gg
\frac{X}{\sqrt x}
\sum_{\substack{m,n\in\mathcal M\\
q(m,n)\le x/(\log X)^2}}
h(mn)q(m,n)^{-1/2}.
\end{equation*}

As in the proof of the first assertion,
\[
h(mn)
\ge
\exp\{-o(A_N)\},
\qquad
A_N:=
\sqrt{\frac{\log N\log_3N}{\log_2N}},
\]
uniformly for $m,n\in\mathcal M$. Moreover, the same localization
estimate gives
\[
\sum_{\substack{m,n\in\mathcal M\\
q(m,n)>x/(\log X)^2}}
q(m,n)^{-1/2}
=
o\left(
\sum_{m,n\in\mathcal M}
q(m,n)^{-1/2}
\right).
\]
Hence
\begin{align*}
\mathcal M_{2,\square}
&\ge
\frac{X}{\sqrt x}
\exp\{-o(A_N)\}
G(\mathcal M)
\nonumber\\
&\ge
\frac{XN}{\sqrt x}
\exp\{(2+o(1))A_N\}.
\end{align*}

We now estimate the non-square contribution. Writing
\[
kmn=s_0s_1^2,
\qquad s_0\ \text{square-free},
\]
and using Lemma~\ref{lem2.2} together with the
submultiplicative properties of $f$ and $g$, we obtain
\[
f(s_0)g(s_1)=X^{o(1)}
\]
uniformly for $k\le z$ and $m,n\in\mathcal M$. Since
\[
0\le
\frac{1-\cos(2\pi k/x)}{k}
\le
\frac2k,
\]
we have
\[
\sum_{k\le z}
\frac{1-\cos(2\pi k/x)}{k}
\ll \log z
=X^{o(1)}.
\]
Therefore
\begin{equation}
\mathcal E_2
\ll
X^{1/2+\eta+o(1)}N^2.
\label{eq:dual-error}
\end{equation}

We now choose $N$ exactly as in the proof of the first assertion.

Suppose first that
\[
x>X^{1/4}.
\]
Take
\[
N=\left\lfloor\frac{\sqrt X}{x}\right\rfloor.
\]
Choose a fixed $\eta$ with
\[
0<\eta<\frac18.
\]
Then, by \eqref{eq:dual-error},
\begin{align*}
\frac{|\mathcal E_2|}
{XN/\sqrt x}
&\ll
X^{-1/2+\eta+o(1)}N\sqrt x
\nonumber\\
&\ll
\frac{X^{\eta+o(1)}}{\sqrt x}
\nonumber\\
&\le
X^{\eta-1/8+o(1)}
=o(1).
\end{align*}
Similarly, the error term in \eqref{eq:dual-M1} is $o(XN)$.
Consequently,
\[
\mathcal M_1\ll XN
\]
and
\[
\mathcal M_2
\ge
\frac{XN}{\sqrt x}
\exp\{(2+o(1))A_N\}.
\]
It follows from \eqref{eq:dual-resonance} that
\begin{equation*}
\max_{d\in\mathcal F_X^-}F_d
\ge
\frac1{\sqrt x}
\exp\{(2+o(1))A_N\}.
\end{equation*}
Since
\[
N\sim T=\frac{\sqrt X}{x},
\]
we have
\[
A_N=(1+o(1))A_T,
\qquad
A_T:=
\sqrt{\frac{\log T\log_3T}{\log_2T}}.
\]

Suppose next that
\[
x\le X^{1/4}.
\]
Fix an arbitrarily small $\delta>0$ and take
\[
N=
\left\lfloor
\frac{X^{1/2-\delta}}{x}
\right\rfloor.
\]
Choose
\[
0<\eta<\frac{\delta}{2}.
\]
Then
\[
\frac{|\mathcal E_2|}
{XN/\sqrt x}
\ll
X^{-1/2+\eta+o(1)}N\sqrt x
\ll
\frac{X^{\eta-\delta+o(1)}}{\sqrt x}
=o(1),
\]
and again the error in $\mathcal M_1$ is $o(XN)$. Hence
\[
\max_{d\in\mathcal F_X^-}F_d
\ge
\frac1{\sqrt x}
\exp\{(2+o(1))A_N\}.
\]
Since
\[
T=\frac{\sqrt X}{x}
\]
and $x\le X^{1/4}$,
\[
\log T\ge\frac14\log X,
\]
while
\[
\log N
=
\log T-\delta\log X+o(\log X)
\ge
(1-4\delta+o(1))\log T.
\]
Thus
\[
A_N
\ge
\left(\sqrt{1-4\delta}+o(1)\right)A_T.
\]
Since the estimate holds for every fixed sufficiently small
$\delta>0$, first letting $X\to\infty$ and then
$\delta\to0^+$ gives
\begin{equation}
\max_{d\in\mathcal F_X^-}F_d
\ge
\frac1{\sqrt x}
\exp\{(2+o(1))A_T\}.
\label{eq:Fd-final}
\end{equation}

Choose $d\in\mathcal F_X^-$ for which
\eqref{eq:Fd-final} holds. Returning to the P\'olya expansion
\eqref{eq:polya-dual}, we obtain
\begin{align*}
\sum_{n\le |d|/x}\chi_d(n)
&=
\frac{\sqrt{|d|}}{\pi}F_d
+
O\left(\sqrt{\frac{X}{x}}\right)
\nonumber\\
&\ge
\sqrt{\frac{|d|}{x}}
\exp\{(2+o(1))A_T\},
\end{align*}
since $|d|\asymp X$ and $A_T\to\infty$. The fixed factor
$1/\pi$ is absorbed into the $o(1)$ in the exponent.

Therefore
\[
 {
\max_{\substack{X<|d|\le2X\\d\in\mathcal F}}
\frac1{\sqrt{|d|/x}}
\sum_{n\le |d|/x}\chi_d(n)
\ge
\exp\left\{
(2+o(1))
\sqrt{
\frac{
\log(\sqrt X/x)\log_3(\sqrt X/x)
}{
\log_2(\sqrt X/x)
}
}
\right\}.
}
\]
This proves the second assertion of Theorem~\ref{thm2}.
	\section*{Acknowledgements}
	Z. Dong is supported by  the National
	Natural Science Foundation of China (Grant No. 	1240011770). W. Wang is supported by the National
	Natural Science Foundation of China (Grant No. 1250012812). H. Zhang is supported by the Fundamental Research Funds for the Central Universities (Grant No. 531118010622), the National
	Natural Science Foundation of China (Grant No. 1240011979) and the Hunan Provincial Natural Science Foundation of China (Grant No. 2024JJ6120).

	\normalem


\begin{thebibliography}{99}
		
			




\bibitem{BT} de la Bret\`{e}che, R.; Tenenbaum, G. {\emph Sommes de G\'{a}l et applications}, {\it Proc. Lond. Math. Soc.}, {\bf 119} (2019), 104--134.
	
	
	

\bibitem{DM} Darbar, P.; Maiti, G. {\emph Large values of quadratic Dirichlet $L$-functions}, {\it Math. Ann.}, {\bf 392} (2025), 4573--4605.

\bibitem{DWZ} Dong, Z.; Wang, W.;  Zhang, H. {\emph Structure of large quadratic character sums}, preprint, arXiv:2306.06355
\bibitem{DSWZ} Dong, Z.; Song, Y.; Wang, R.; Zhao, S. {\emph Note  on large quadratic character sums}, preprint, arXiv:2510.09005
\bibitem{DWWZ}Dong, Z.; Wang, R.; Wang, W.; Zhang, H. {\emph Large quadratic character sums revisited}, preprint, arXiv:2512.24147
\bibitem{DZ} Dong, Z.;  Zhang, Y. {\emph Large quadratic character sums}, preprint, arXiv:2509.07651
\bibitem{GS} Granville, A.; Soundararajan, K. {\emph The Distribution of values of $L(1,\chi_d)$}, {\it Geom. Funct. Anal.}, {\bf 13} (2003), 992--1028.
 


\bibitem{GS01} Granville, A.; Soundararajan K. {\emph Large character sums}, {\it J. Amer. Math. Soc.}, {\bf 14} (2001), 365--397.

\bibitem{Hil} Hilberdink, T. {\emph An arithmetical mapping and applications to results for the Riemann zeta
function}, {\it Acta Arith.}, {\bf 139} (2009), 341--367.
 
\bibitem{Hou} Hough, B. {\emph The resonance method for large character sums}, {\it  Mathematika},  {\bf 59},(2013), 87-118



\bibitem{Lam} Lamzouri, Y. {\emph The distribution of large quadratic character sums and applications}, {\it Algebra $\&$ Number Theory}, {\bf 18} (2024), 2091--2131.

\bibitem{MVbook} Montgomery H. L.;  Vaughan  R.C.
\emph{Multiplicative Number Theory I: Classical Theory},
Cambridge Studies in Advanced Mathematics, Vol. 97. Cambridge University Press, 2006.
\bibitem{Mun} Munsch, M.  {\emph The maximum size of short character sums}, {\it Ramanujan J.} {\bf 53} (2020), 27–-38.

\bibitem{Paley}
 Paley, R. E. A. C. {\emph A theorem on characters}, {\it J. Lond. Math. Soc.}, {\bf 1} (1932),
28--32. 

\bibitem{Sound}
 Soundararajan, K. {\emph Extreme values of zeta and $L$-functions}, {\it Math. Ann.}, {\bf 342} (2008),
67--86. 





		
	\end{thebibliography}
\end{document}